\documentclass[a4paper, 11pt]{amsproc}
\usepackage{amsmath}    % Extends math environments
\usepackage{amsthm}     % Theorem environments
\usepackage{amssymb}    % Additional math symbols
\usepackage{mathtools}  % Extends amsmath
\usepackage{amstext}    % Text within math

\usepackage{mathrsfs}   % Script font (for math)
\usepackage{bm}         % Bold symbols (for math)
\usepackage{shuffle}    % Shuffle product symbol
\usepackage[all]{xy}    % Diagrams

\usepackage{amsrefs}    % AMS-style citations and bibliography
\usepackage{ascmac}     % Boxed items (e.g., notes, algorithms)
\usepackage{comment}    % Commenting out sections
\usepackage{ytableau}   % Young tableaux
\usepackage{here}       % Control float placement
\usepackage{time}       % Current time
\usepackage{fullpage}   % Use full page
\usepackage{xcolor}     % Colors

\usepackage[utf8]{inputenc} % Set input encoding to UTF-8

\usepackage{hyperref}   % Hyperlinks
\hypersetup{
 colorlinks=true,       % false: boxed links; true: colored links
 linkcolor=blue,        % color of internal links
 citecolor=blue,        % color of links to bibliography
 filecolor=blue,        % color of file links
 urlcolor=blue          % color of external links
}

\usepackage[capitalize,nameinlink,noabbrev,nosort]{cleveref}

\DeclareFontEncoding{OT2}{}{}
\DeclareFontSubstitution{OT2}{cmr}{m}{l}
\DeclareFontFamily{OT2}{cmr}{\hyphenchar\font45 }
\DeclareFontShape{OT2}{cmr}{m}{l}{%
  <5><6><7><8><9>gen*wncyr%
  <10><10.95><12><14.4><17.28><20.74><24.88>wncyr10}{}
\DeclareMathAlphabet{\mathcyr}{OT2}{cmr}{m}{l}
\DeclareMathAlphabet{\mathcyb}{OT2}{cmr}{b}{l}
\SetMathAlphabet{\mathcyr}{bold}{OT2}{cmr}{b}{l}

\makeatletter
\@namedef{subjclassname@2020}{%
  \textup{2020} Mathematics Subject Classification}
\makeatother

\newtheorem{theoremcounter}{Theorem Counter}[section]
\theoremstyle{definition}

\newtheorem{remark}[theoremcounter]{Remark}

\theoremstyle{plain}
\newtheorem{lemma}[theoremcounter]{Lemma}

\newtheorem{corollary}[theoremcounter]{Corollary}

\newtheorem{theorem}[theoremcounter]{Theorem}

\theoremstyle{definition}

\theoremstyle{plain}

\numberwithin{equation}{section}

\newcommand{\Q}{\mathbb{Q}}

\newcommand{\C}{\mathbb{C}}

\allowdisplaybreaks[1]

\begin{document}

\title{On the mean values of the Barnes multiple zeta function}

\author{Takashi Miyagawa}
\address[Takashi Miyagawa]{Onomichi City University,  1600-2 Hisayamada-cho, Onomichi, Hiroshima, 722-8506, Japan} 
\email{miyagawa@onomichi-u.ac.jp}

\author{Hideki Murahara}
\address[Hideki Murahara]{The University of Kitakyushu,  4-2-1 Kitagata, Kokuraminami-ku, Kitakyushu, Fukuoka, 802-8577, Japan}
\email{hmurahara@mathformula.page}

%%%%%%%%%%%%%%%%%%%%%%%%%%%%%%%%%%%%%%%%%%%%%%%%%%%%%%%%%%%%
\subjclass[2020]{Primary 11M32}

%%%%%%%%%%%%%%%%%%%%%%%%%%%%%%%%%%%%%%%%%%%%%%%%%%%%%%%%%%%%
\begin{abstract}
The asymptotic behavior of the mean values of multiple zeta functions is of significant interest due to its close connection with the Riemann zeta function.  
In this paper, we establish asymptotic behavior of the mean square values of Barnes multiple zeta functions.
\end{abstract}

%%%%%%%%%%%%%%%%%%%%%%%%%%%%%%%%%%%%%%%%%%%%%%%%%%%%%%%%%%%%
\keywords{Riemann zeta function, Barnes multiple zeta function, mean value theorem}

%%%%%%%%%%%%%%%%%%%%%%%%%%%%%%%%%%%%%%%%%%%%%%%%%%%%%%%%%%%%
\maketitle

%%%%%%%%%%%%%%%%%%%%%%%%%%%%%%%%%%%%%%%%%%%%%%%%%%%%%%%%%%%%
\section{Introduction}
Let $r $ be a positive integer, $s=\sigma+it$ a complex variable, $a$ a real positive parameter, and $w_j\;(j=1,\dots,r)$ complex parameters located in the same half-plane, separated by a straight line through the origin. 
The Barnes multiple zeta function, denoted by $\zeta_r(s,a,(w_1, \dots,w_r))$ and introduced in \cites{Barnes1899, Barnes1901, Barnes1904}, is defined as  
\begin{align}\label{zeta_r}
 \zeta_r (s,a,(w_1,\dots,w_r))
 =\sum_{m_1=0}^\infty \cdots \sum_{m_r=0}^\infty
  \frac{1}{(a+m_1 w_1+\cdots+m_r w_r)^s}.
\end{align}
The series on the right-hand side converges absolutely for $\mathrm{Re}(s) > r$ and can be meromorphically continued to the entire complex $s$-plane. 
The function has simple poles located at $s=j$ for $j=1,\dots,r$.

The Barnes multiple zeta function was originally introduced by Barnes in the course of developing the theory of multiple gamma functions, which can be regarded as higher-dimensional analogues of the classical gamma function.  
This construction is based on Lerch's formula for the Hurwitz zeta function, which expresses the logarithm of the classical gamma function as
\[
\log \Gamma(a) = \zeta'(0,a) + \frac{1}{2}\log 2\pi.
\]
In particular, the multiple gamma function can be defined in terms of the derivative of the Barnes multiple zeta function at $s=0$ as
\[
\log \Gamma_r(a,\mathbf{w})
= \left. \frac{\partial}{\partial s} \zeta_r(s,a,\mathbf{w}) \right|_{s=0}
+ \rho_r(\mathbf{w}),
\]
where $\rho_r(\mathbf{w})$ is a normalization constant depending only on $\mathbf{w}$.  
Since then, Barnes-type zeta functions have been studied extensively in connection with special functions and analytic number theory.

Moreover, Barnes multiple zeta functions are related to several areas such as spectral zeta functions and combinatorial geometry, including the study of lattice point enumeration in polyhedral cones.  
In particular, the Barnes double zeta function has been used in problems related to counting lattice points in planar regions and in the study of sums of the form
\[
\sum_{n \le x} \left(\theta n - \lfloor \theta n \rfloor - \frac{1}{2}\right),
\]
where $\theta$ is a fixed irrational number, which arise in analytic number theory (see, e.g., \cite{HardyLittlewood1922first}).  
These connections indicate that the study of their analytic properties is of independent interest.

Research on the mean values of multiple zeta functions began with the study of the squared mean value of the Euler-Zagier double zeta function
\[
 \zeta_{EZ,2}(s_1,s_2)
 =\sum_{m_1=1}^\infty \sum_{m_2=1}^\infty \frac{1}{{m_1}^{s_1}(m_1+m_2)^{s_2}} 
\]
by Matsumoto and Tsumura \cite{MatsumotoTsumura2015}.
Related studies include those by Ikeda, Matsuoka, and Nagata \cite{IkedaMatsuokaNagata2015},
Ikeda, Kiuchi, and Matsuoka \cites{IkedaKiuchiMatsuoka2016, IkedaKiuchiMatsuoka2017},
Kiuchi and Minamide \cite{KiuchiMinamide2016}, and Banerjee, Minamide, and Tanigawa \cite{BanerjeeMinamideTanigawa2021}.
Moreover, Okamoto and Onozuka \cite{OkamotoOnozuka2015} provided results on the squared mean value of the Mordell-Tornheim double zeta function
\[
 \zeta_{MT, 2}(s_1, s_2, s_3) 
 =\sum_{m_1=1}^\infty\sum_{m_2=1}^\infty\frac{1}{m_1^{s_1}m_2^{s_2}(m_1+m_2)^{s_3}},
\]
while Toma \cite{Toma2023} obtained similar results for the Apostol-Vu double zeta function
\[
 \zeta_{AV, 2}(s_1, s_2, s_3) 
 =\sum_{1\le m_2<m_1}
 \frac{1}{m_1^{s_1}m_2^{s_2}(m_1+m_2)^{s_3}}.
\]

Boldface symbols are used to represent index tuples:  
\[
 \mathbf{m}=(m_1,\dots,m_r), \quad
 \mathbf{n}=(n_{1},\dots,n_{r}), \quad
 \mathbf{w}=(w_1,\dots,w_r).
\]
We assume that $w_1,\dots,w_r\in \mathbb{R}_{>0}$ throughout this paper.
Using this notation, \eqref{zeta_r} can be rewritten as  
\begin{align}
 \zeta_r(s,a,\mathbf{w})
 =\sum_{m_1,\dots,m_r\ge 0}
 \frac{1}{(a+\mathbf{m}\cdot\mathbf{w})^s}.
\end{align}
We introduce the following function, which will appear in our main results: 
\[
 \tilde{\zeta_{r}}(\sigma,a,\mathbf{w})
 =\sum_{\substack{m_1,\dots,m_r\ge0\\n_{1},\dots,n_{r}\ge0\\
\mathbf{m}\cdot\mathbf{w}=\mathbf{n}\cdot\mathbf{w}}}
  \frac{1}{(a+\mathbf{m}\cdot\mathbf{w})^{\sigma}(a+\mathbf{n}\cdot\mathbf{w})^{\sigma}}.
\]
The condition $\mathbf{m}\cdot\mathbf{w}=\mathbf{n}\cdot \mathbf{w}$ implies $\mathbf{m}=\mathbf{n}$ if the parameters $w_1,\dots,w_r$ are linearly independent over $\mathbb{Q}$. 
Consequently, we have 
\begin{align*}
 \tilde{\zeta_{r}}(\sigma,a,\mathbf{w})
 &=\sum_{\substack{m_1,\dots,m_r\ge0\\n_{1},\dots,n_{r}\ge0\\\mathbf{m}=\mathbf{n}}}
  \frac{1}{(a+\mathbf{m}\cdot\mathbf{w})^{\sigma}(a+\mathbf{n}\cdot\mathbf{w})^{\sigma}}
 =\sum_{m_1,\dots,m_r\ge0}
  \frac{1}{(a+\mathbf{m}\cdot\mathbf{w})^{2\sigma}}\\
 &=\zeta_{r}(2\sigma,a,\mathbf{w}).
\end{align*}

We denote by
$
d=\dim_{\mathbb{Q}}\langle w_1,\dots,w_r\rangle
$
the dimension of the $\mathbb{Q}$-vector space generated by
$w_1,\dots,w_r$, which we call the $\mathbb{Q}$-rank of
$\mathbf{w}=(w_1,\dots,w_r)$.

\begin{theorem}\label{main1}
Assume that $d=1$.
 For $s=\sigma+it\in\mathbb{C}$ with $\sigma>r$, we have
 \[
  \int_{1}^{T}|\zeta_{r}(\sigma+it,a,\mathbf{w})|^{2}dt
  =\tilde{\zeta_{r}}(\sigma,a,\mathbf{w})T+O(1)
 \]
 as $T\rightarrow+\infty$.
\end{theorem}

When approximating the Riemann zeta function by a finite sum, the following formula holds:
Let $\sigma_0>0$, $x \ge 1$, and $C > 1$.  
If $s=\sigma+it \in \mathbb{C}$ satisfies $\sigma_0 \le \sigma \le 2$ and $|t| \le 2\pi x/ C$, then  
\[
 \zeta(s)=\sum_{1\le m\le x} \frac{1}{m^s}-\frac{x^{1-s}}{1-s}+O(x^{-\sigma})
\]
as $x\rightarrow \infty$.
Similarly, an approximation formula for $\zeta_r(s,a,\mathbf{w})$ is given by the following expression.
\begin{theorem}\label{main2}
Let $r-1<\sigma_1<\sigma_2$, $x\ge 1$, and $C>1$. 
If $s=\sigma+it \in \C$ with $\sigma_1<\sigma<\sigma_2$ and $|t|\le 2\pi x/C$, then
\begin{align*} 
 \zeta_r(s,a,\mathbf{w})
 &=\sum_{0\le m_1\le x}
  \cdots 
  \sum_{0\le m_r\le x}
  \frac{1}{(a+\mathbf{m}\cdot\mathbf{w})^{s}}\\
 &\quad -
 \sum_{\substack{E\subseteq\{w_1,\dots,w_r\}\\E\ne\emptyset}}
 (-1)^{\#E}
 \frac{(a+x\sum_{ e\in E }e)^{r-s}}{(s-1)\cdots (s-r) w_1\cdots w_{r}}
 +O(x^{r-1-\sigma})
\end{align*}
as $x\rightarrow \infty$.
\end{theorem}
\begin{corollary}[Theorem 3 in \cite{Miyagawa2018}]
 Let $1<\sigma_1<\sigma_2,\ x\ge 1$, and $ C>1 $.
 If $s=\sigma+it \in \mathbb{C} $ with $\sigma_1<\sigma<\sigma_2$ and $|t|\le 2\pi x/C$, then 
 \begin{align*}
 \zeta_2(s,a,(w_1,w_2)) 
 &=\sum_{0\le m_1 \le x} \sum_{0\le m_{2} \le x} 
   \frac{1}{(a+m_1 w_1+m_{2} w_2)^s}\\
 &\quad 
  +\frac{(a+x w_1)^{2-s} 
  +(a+x w_2)^{2-s} 
  -(a+x w_1+x w_2)^{2-s}}{(s-1)(s-2)w_1 w_2} 
  +O(x^{1-\sigma} )  
 \end{align*}
as $x\rightarrow \infty$.  
\end{corollary}

Using Theorem \ref{main2}, we obtain the mean value theorem.
\begin{theorem}\label{main3}
Assume that $d=1$.
For $s=\sigma+it \in\mathbb{C}$ with $r-1<\sigma\le r$, we have
\begin{align*}
 \int_1^T |\zeta_r(\sigma+it,a,\mathbf{w})|^2 \,dt
 =\tilde{\zeta_{r}}(\sigma,a,\mathbf{w})T 
 +
  \begin{cases}
   O(T^{2r-2\sigma} \log{T}) 
   &\text{if }  r-1/2<\sigma \le r-1/4,	\\
   O({T}^{1/2})
   &\text{if } r-1/4<\sigma \le r,
  \end{cases}
\end{align*}
and 
\begin{align*}
 \int_1^T |\zeta_r(\sigma+it,a,\mathbf{w})|^2 \,dt 
 \ll T^{2r-2\sigma} \log T \quad \text{if } r-1<\sigma\le r-1/2
\end{align*}
as $T \rightarrow +\infty$.
\end{theorem}
\bigskip

\begin{remark}
The case $r=2$ of Theorems \ref{main1} and \ref{main3}
corresponds to Theorems 1 and 2 in \cite{Miyagawa2018}, respectively.
However, in \cite{Miyagawa2018}, the condition $d=1$, namely that the parameters are linearly dependent over $\Q$,
is not explicitly stated.
In fact, this condition is necessary for the validity of those results.
Therefore, in the present paper we explicitly impose the assumption $d=1$ in Theorems \ref{main1} and \ref{main3}.
\end{remark}

%%%%%%%%%%%%%%%%%%%%%%%%%%%%
\section{Proof of Theorem \ref{main1}} 
In this section, we give a proof of Theorem \ref{main1}.
\begin{proof}[Proof of Theorem \ref{main1}] 
% We first consider the case where $w_1,\dots,w_r$ are linearly dependent over $\mathbb{Q}$.  
%
Since 
\begin{align*}
|\zeta_{r}(s,a,\mathbf{w})|^{2} &=\zeta_{r}(s,a,\mathbf{w})\overline{\zeta_{r}(s,a,\mathbf{w})}\\
 &=\tilde{\zeta_{r}}(\sigma,a,\mathbf{w})
  +\sum_{\substack{m_1,\dots,m_r\ge0\\n_{1},\dots,n_{r}\ge0\\
\mathbf{m}\cdot\mathbf{w}\ne\mathbf{n}\cdot\mathbf{w}}}
  \frac{1}{(a+\mathbf{m}\cdot\mathbf{w})^{\sigma}(a+\mathbf{n}\cdot\mathbf{w})^{\sigma}}
  \left(
   \frac{a+\mathbf{n}\cdot\mathbf{w}}{a+\mathbf{m}\cdot\mathbf{w}}
  \right)^{it},
\end{align*}
we have 
\begin{align*}
 &\int_{1}^{T}|\zeta_{r}(s,a,\mathbf{w})|^{2}dt\\
 &=\tilde{\zeta_{r}}(\sigma,a,\mathbf{w})(T-1)\\
  &\quad
   +\sum_{\substack{m_1,\dots,m_r\ge0\\n_{1},\dots,n_{r}\ge0\\
\mathbf{m}\cdot\mathbf{w}\ne\mathbf{n}\cdot\mathbf{w}}}
   \frac{1}{(a+\mathbf{m}\cdot\mathbf{w})^{\sigma}(a+\mathbf{n}\cdot\mathbf{w})^{\sigma}}
   \cdot
   \frac{e^{iT\log\{(a+\mathbf{n}\cdot\mathbf{w})/(a+\mathbf{m}\cdot\mathbf{w})\}}-e^{i\log\{(a+\mathbf{n}\cdot\mathbf{w})/(a+\mathbf{m}\cdot\mathbf{w})\}}}{i\log\{(a+\mathbf{n}\cdot\mathbf{w})/(a+\mathbf{m}\cdot\mathbf{w})\}}.
\end{align*}
The second term on the right-hand side is estimated as
\begin{align*}
 &\ll\sum_{\substack{m_1,\dots,m_r\ge0\\
n_{1},\dots,n_{r}\ge0\\
\mathbf{m}\cdot\mathbf{w}<\mathbf{n}\cdot\mathbf{w}
}
}\frac{1}{(a+\mathbf{m}\cdot\mathbf{w})^{\sigma}(a+\mathbf{n}\cdot\mathbf{w})^{\sigma}}\cdot\frac{1}{\log\{(a+\mathbf{n}\cdot\mathbf{w})/(a+\mathbf{m}\cdot\mathbf{w})\}}\\
 &=
  \left(
  \sum_{\substack{m_1,\dots,m_r\ge0\\
n_{1},\dots,n_{r}\ge0\\
a+\mathbf{m}\cdot\mathbf{w}<a+\mathbf{n}\cdot\mathbf{w}<2(a+\mathbf{m}\cdot\mathbf{w})}}
  +\sum_{\substack{m_1,\dots,m_r\ge0\\n_{1},\dots,n_{r}\ge0\\
a+\mathbf{n}\cdot\mathbf{w}\ge2(a+\mathbf{m}\cdot\mathbf{w})}}
  \right)\\
 &\qquad \qquad
 \frac{1}{(a+\mathbf{m}\cdot\mathbf{w})^{\sigma}(a+\mathbf{n}\cdot\mathbf{w})^{\sigma}}\cdot\frac{1}{\log\{(a+\mathbf{n}\cdot\mathbf{w})/(a+\mathbf{m}\cdot\mathbf{w})\}},
\end{align*}
say $V_{1}+V_{2}$. 
Since
\[
\log\left(\frac{a+\mathbf{n}\cdot\mathbf{w}}{a+\mathbf{m}\cdot\mathbf{w}}\right)\ge \log 2
\]
for all terms appearing in $V_2$, and since $\sigma>r$, we have $V_2=O(1)$.

Now, let us consider $V_{1}$. 
The range of $n_{r}$ satisfying
the inequalities $a+\mathbf{m}\cdot\mathbf{w}<a+\mathbf{n}\cdot\mathbf{w}<2(a+\mathbf{m}\cdot\mathbf{w})$
of the summation condition of $V_{1}$ is 
\begin{align}\label{range_of_nr}
 \frac{1}{w_{r}}
 \sum_{j=1}^{r-1} (m_{j}-n_{j})w_{j} 
 +m_r<n_{r}<\frac{a}{w_{r}}+\frac{1}{w_{r}}
 \sum_{j=1}^{r-1}(2m_{j}-n_{j})w_{j}+2m_r. 
\end{align}
Let 
$\varepsilon=\varepsilon(m_1,\dots,m_r,n_{1},\dots,n_{r-1})$, 
$\delta=\delta(a,m_1,\dots,m_r,n_{1},\dots,n_{r-1})$
be the quantities satisfying $0\le\varepsilon,\delta<1$ and 
\begin{align*}
 \frac{1}{w_{r}}\sum_{j=1}^{r-1}
 (m_{j}-n_{j})w_{j}+m_r+\varepsilon
 &\in\mathbb{Z},\\
 %%%
 \frac{a}{w_{r}}
 +\frac{1}{w_{r}}\sum_{j=1}^{r-1}
  (2m_{j}-n_{j})w_{j}+2m_r-\delta
  &\in\mathbb{Z}. 
\end{align*}
Then 
\begin{align*}
 K 
 &\coloneqq\left(
  \frac{a}{w_{r}}
  +\frac{1}{w_{r}}
  \sum_{j=1}^{r-1}(2m_{j}-n_{j})w_{j}
  +2m_r-\delta
  \right)
  -\left(
   \frac{1}{w_{r}}
   \sum_{j=1}^{r-1}(m_{j}-n_{j})w_{j}+m_r
   +\varepsilon
  \right)\\
 &=\frac{a}{w_{r}}
 +\frac{1}{w_{r}}
 \sum_{j=1}^{r-1}m_{j}w_{j}+m_r
 -\varepsilon-\delta
\end{align*}
is an integer, and, by \eqref{range_of_nr}, $n_{r}$ can be rewritten as 
\[
 n_{r}
 =\frac{1}{w_{r}}
  \sum_{j=1}^{r-1}(m_{j}-n_{j})w_{j}
  +m_r+\varepsilon+k
\]
for some $k=0,1,\dots,K$. 
From this, we have $a+\mathbf{n}\cdot\mathbf{w}=a+\mathbf{m}\cdot\mathbf{w}+w_{r}\varepsilon+w_{r}k$,
and then 
\begin{align*}
 \log\frac{a+\mathbf{n}\cdot\mathbf{w}}{a+\mathbf{m}\cdot\mathbf{w}} &=\log\left(1+\frac{w_{r}\varepsilon+w_{r}k}{a+\mathbf{m}\cdot\mathbf{w}}\right)\asymp\frac{w_{r}(k+\varepsilon)}{a+\mathbf{m}\cdot\mathbf{w}}.
\end{align*}
Since $d=1$, there exist $\lambda>0$ and positive integers
$p_1,\dots,p_r$ such that $w_j=\lambda p_j$ $(1\le j\le r)$.
Hence
\begin{equation}\label{linear}
 \frac{1}{w_r}\sum_{j=1}^{r-1}(m_j-n_j)w_j+m_r
 =
 \frac{1}{p_r}\sum_{j=1}^{r-1}(m_j-n_j)p_j+m_r,
\end{equation}
so $\varepsilon$ is either $0$ or satisfies $\varepsilon\ge 1/p_r$.
Moreover, since $\mathbf m\cdot\mathbf w\ne \mathbf n\cdot\mathbf w$, we have
$k+\varepsilon>0$. Therefore,
\begin{equation}\label{<<log(2+K)}
    \sum_{\substack{0\le k\le K\\ k+\varepsilon>0}} \frac{1}{k+\varepsilon}
 \ll \log(2+K).
\end{equation}
Thus we have
\begin{align*}
 V_{1} 
 &=\sum_{\substack{m_1,\dots,m_r\ge0\\
 n_{1},\dots,n_{r}\ge0\\
 a+\mathbf{m}\cdot\mathbf{w}<a+\mathbf{n}\cdot\mathbf{w}<2(a+\mathbf{m}\cdot\mathbf{w})}}
 \frac{1}{(a+\mathbf{m}\cdot\mathbf{w})^{\sigma}
 (a+\mathbf{n}\cdot\mathbf{w})^{\sigma}}
 \cdot
 \frac{1}{\log\{(a+\mathbf{n}\cdot\mathbf{w})/(a+\mathbf{m}\cdot\mathbf{w})\}}\\
 %%%
 &\ll
 \sum_{m_1,\dots,m_r\ge0}
 \sum_{\substack{K\asymp 1+m_1+\cdots+m_r \\ 0 \le n_1, \dots , n_{r-1} \ll K}} 
 \sum_{k=0}^K
 \frac{1}{(a+\mathbf{m}\cdot\mathbf{w})^{\sigma}(a+\mathbf{m}\cdot\mathbf{w}+w_{r}k)^{\sigma}}
 \cdot
 \frac{a+\mathbf{m}\cdot\mathbf{w}}{w_{r}(k+\varepsilon)}\\
 %%%
 &\ll
 \sum_{m_1,\dots,m_r\ge0}
 \frac{K^{r-1} \log K}{(a+\mathbf{m}\cdot\mathbf{w})^{2\sigma-1}}.
\end{align*}
Using  
$K\asymp a+\mathbf{m}\cdot\mathbf{w}\asymp1+m_1+\cdots+m_r$, 
\begin{gather*}
 V_1
 \ll
 \sum_{m_1,\dots,m_r\ge0}
 \frac{\log(1+m_1+\cdots +m_r)}{(1+m_1+\cdots +m_r)^{2\sigma-r}}
 %%%
 =\sum_{N=0}^\infty 
 \frac{\log(1+N)}{(1+N)^{2\sigma-r}} \binom{N+r-1}{r-1}.
\end{gather*}
Since $\binom{N+r-1}{r-1}\asymp N^{r-1}$ and $\sigma>r$, we have
\begin{gather*}
 V_1
 %%%
 \ll
 \sum_{N=0}^\infty 
 \frac{\log(1+N)}{(1+N)^{2\sigma-2r+1}}
 %%%
 \ll 1.
\end{gather*}
This finishes the proof.
\end{proof}

\bigskip

 \begin{remark}
 The assumption $d=1$ is essential in the above argument.
 Indeed, the estimate \eqref{<<log(2+K)}
 requires a uniform lower bound for $\varepsilon$.
% %
 If $d=1$, then there exist $\lambda>0$ and positive integers
 $p_1,\dots,p_r$ such that $w_j=\lambda p_j$ $(1\le j\le r)$.
 Hence the quantity \eqref{linear} 
 takes values in a lattice with denominator $p_r$,
 so that $\varepsilon$ is either $0$ or satisfies
 $\varepsilon\ge 1/p_r$.
% %
 On the other hand, if $d\ge2$, the numbers
 \[
 \frac{1}{w_r}\sum_{j=1}^{r-1}(m_j-n_j)w_j+m_r
 \]
 can approach integers arbitrarily closely as
 $(m_1,\dots,m_r,n_1,\dots,n_{r-1})$ vary.
 In this case $\varepsilon$ may become arbitrarily small,
 and therefore the bound \eqref{<<log(2+K)} cannot be obtained uniformly.
 \end{remark}

%%%%%%%%%%%%%%%%%%%%%%%%%%%%
\section{Proof of Theorem \ref{main2}}
\begin{lemma}[Lemma 4.10 in \cite{Titchmarsh1986}]\label{exp sum}
Let $f(\xi)$ be a real function with a continuous and steadily increasing derivative $f'(\xi)$ in $(p,q)$. 
Let $g(\xi)$ be a real positive decreasing function with a continuous derivative $g'(\xi)$, satisfying that $|g'(\xi)|$ is steadily decreasing.

Then we have
\begin{align}\label{exponential-sum}
 \begin{split}
 \sum_{p<n\le q} g(n) e^{2\pi i f(n)}
 &=\sum_{\substack{ \nu \in \mathbb{Z} \\ f'(p)-\varepsilon<\nu<f'(q)+\varepsilon }}
  \int_p^q g(\xi) e^{2\pi i (f(\xi)+\nu \xi)} d\xi \\
 &\quad
  +O(g(p) \log(f'(q)-f'(p)+2) )+O(|g'(p)|)
 \end{split}
\end{align}
for an arbitrary $\varepsilon \in (0,1)$.
\end{lemma}

\begin{lemma} \label{lemA}
Let $p_j, q_j$ be positive numbers such that $0 \le p_j < q_j$ for each $j=1,\dots,r$, and suppose that $d$ of $q_j$ tend to infinity, where $d=0,1,\dots,r-1$.  
Let $s=\sigma+it \in \mathbb{C}$ with $\sigma > r - 1$, 
and let $ P=\max \{ p_j \ | \ 1 \le j \le r \} $ and $C > 1$. If $|t| \le 2\pi P / C$, 
then we have
 \begin{align}
 \begin{split}
  &\sum_{p_{1}\le m_1 \le q_{1}}
   \cdots
   \sum_{p_{r}\le m_r \le q_{r}}
   \frac{1}{(a+\mathbf{m}\cdot\mathbf{w})^s}\\
 &=\frac{1}{(s-1)\cdots (s-r) w_1\cdots w_{r}}
  \sum_{\substack{b_{j}\in\{p_{j},q_{j}\,|q_{j}<\infty\}\\j=1,\dots,r}}
  (-1)^{\#\{b_{j}=q_{j}\}}
  (a+b_{1}w_1+\cdots+b_{r}w_{r})^{r-s}\\
  &\quad
  % +\sum_{l=1}^{r}
  %  \sum_{\substack{c_{j}=\min\{1+p_{j},q_{j}\,|q_{j}<\infty\}\\j=1,\dots,l}}
  %  O(c_1\cdots c_{l-1}\cdot c_l^{r-l-\sigma}),
  +O\left(\sum_{\substack{b_{j}\in\{p_{j},q_{j}\,|q_{j}<\infty\}\\j=1,\dots,r}}
  (a+b_{1}+\cdots+b_{r})^{r-1-\sigma}
  \right),
 \end{split}\label{sum_{p,q}}
 \end{align}
 where the $O$-constants depend on $|s|$ and $a, w_1,\dots,w_r$.
\end{lemma}
\begin{proof}
We prove the lemma by induction on $r$.
First, we consider the sum
\[
 \sum_{p_1 \le m_1 \le q_1} \frac{1}{(a+m_1 w_1)^s}
 =\sum_{p_1 \le m_1 \le q_1}
  (a+m_1 w_1)^{-\sigma} e^{-it \log(a+m_1 w_1)}.
\]
We apply Lemma \ref{exp sum} to the exponential sum
\[
\sum_{p_1 \le m_1 \le q_1} g(m_1) e^{2\pi i f(m_1)},
\]
where we set 
\[
 g(\xi)
 =(a+\xi w_1)^{-\sigma}, \quad 
 f(\xi)
 =-\frac{t}{2\pi} \log(a+\xi w_1).
\]
Since $t \in \mathbb{R}$ is fixed, we see that $f'(\xi)=-t w_1/{2\pi(a+\xi w_1)}$ is a real, continuous, strictly increasing function of $\xi \in (p_1, q_1)$.
Then $f'(\xi)$ is steadily increasing and negative.
Moreover, $g(\xi)$ is real, positive, and strictly decreasing in $\xi$, with derivative
\[
 g'(\xi)=-\sigma w_1 (a+\xi w_1)^{-\sigma-1},
\]
which is continuous and steadily decreasing in absolute value.
Thus, the hypotheses of Lemma \ref{exp sum} are satisfied.

In the case where $q_j<\infty \ (j=1,\dots,r)$, applying the lemma yields
\begin{align*}
 \sum_{p_1 \le m_1 \le q_1} 
 \frac{1}{(a+m_1 w_1)^s}
 &=\sum_{\substack{ \nu \in \mathbb{Z} \\ f'(p_1)-\varepsilon<\nu<f'(q_1)+\varepsilon }}
 \int_{p_1}^{q_1} g(\xi) e^{2\pi i (f(\xi)+\nu \xi)} \,d\xi
 +O\left( (a+p_1)^{-\sigma} \right).
\end{align*} 
Since \[
 |f'(\xi)|
 =\left| \frac{tw_1}{2\pi(a+\xi w_1)} \right|
 \le 
  \frac{1}{2\pi} \left| \frac{2\pi P}{C} 
  \cdot 
  \frac{w_1}{a+\xi w_1} \right|
 \le \frac{1}{C}<1,
\]
we see that for sufficiently large $q_1$, it is possible to choose $\varepsilon$ so that 
$-1< f'(p_1)-\varepsilon<0<f'(q_1)+\varepsilon<1$
holds, in which case only the term with $\nu=0$ contributes to the sum on the right-hand side.
Then we have
\begin{align}
 \sum_{p_1 \le m_1 \le q_1} 
 \frac{1}{(a+m_1 w_1)^s}
 &=\int_{p_1}^{q_1} g(\xi) e^{2\pi i f(\xi)} \,d\xi
 +O\left( (a+p_1)^{-\sigma} \right) \nonumber\\
 &=\int_{p_1}^{q_1} \frac{1}{(a+\xi w_1)^s}d\xi +O\left( (a+p_1)^{-\sigma} \right) \nonumber\\
 &=\frac{ (a+p_1 w_1)^{1-s}-(a+q_1 w_1)^{1-s}}{(s-1) w_1} 
  +O\left((a+p_1)^{-\sigma} \right). \label{sum(1/(a+mw))}
\end{align} 

To prove the general case, we proceed by induction.
Suppose that the lemma holds for a fixed integer $r=k\ge1$; that is,
 \begin{align*}
 \begin{split}
 &\sum_{p_1 \le m_1 \le q_1} 
 \cdots 
 \sum_{p_k \le m_k \le q_k}
 \frac{1}{(a+m_{1}w_1+\cdots+m_{k}w_{k})^s}\\
 &=\frac{1}{(s-1)\cdots (s-k) w_1\cdots w_{k}}
  \sum_{\substack{b_{j}\in\{p_{j},q_{j}\}\\j=1,\dots,k}}
  (-1)^{\#\{b_{j}=q_{j}\}}
  (a+b_{1}w_1+\cdots+b_{k}w_{k})^{k-s}\\
  &\quad
  %+\sum_{l=1}^{k}
  % \sum_{\substack{c_{j}=\min\{1+p_{j},q_{j}\}\\j=1,\dots,l}}
  % O(c_1\cdots c_{l-1}\cdot c_l^{k-l-\sigma}).
  +O\left(\sum_{\substack{b_{j}\in\{p_{j},q_{j}\}\\j=1,\dots,k}}
  (a+b_{1}+\cdots+b_{k})^{k-1-\sigma}
  \right).
 \end{split}
 \end{align*}
We show that the result also holds for $r=k+1$. Consider the $(k+1)$-fold sum
\begin{align*}
 F
 &\coloneqq
 \sum_{p_1 \le m_1 \le q_1} 
 \cdots 
 \sum_{p_{k+1} \le m_{k+1} \le q_{k+1}}
 \frac{1}{(a+m_{1}w_1+\cdots+m_{k}w_{k}+m_{k+1}w_{k+1})^s}\\
 &=\sum_{p_{k+1} \le m_{k+1} \le q_{k+1}}
 \left(
 \sum_{p_1 \le m_1 \le q_1} 
 \cdots 
 \sum_{p_{k} \le m_{k} \le q_{k}}
 \frac{1}{((a+m_{k+1}w_{k+1})+m_{1}w_1+\cdots+m_{k}w_{k})^s}
 \right).
\end{align*}
By the inductive hypothesis, we have
\begin{align*}
 F
 &=\frac{1}{(s-1)\cdots (s-k) w_1\cdots w_{k}}
  \sum_{\substack{b_{j}\in\{p_{j},q_{j}\}\\j=1,\dots,k}}
  (-1)^{\#\{b_{j}=q_{j}\}}\\
  &\quad\cdot
  \sum_{p_{k+1} \le m_{k+1} \le q_{k+1}}
  ((a+m_{k+1}w_{k+1})+b_{1}w_1+\cdots+b_{k}w_{k})^{k-s}\\
  &\quad
  +O\left(\sum_{\substack{b_{j}\in\{p_{j},q_{j}\}\\j=1,\dots,k}}
  \sum_{p_{k+1}\le m_{k+1}\le q_{k+1}}((a+m_{k+1})+b_{1}+\cdots+b_{k})^{k-1-\sigma}
  \right)\\
 &
 =\frac{1}{(s-1)\cdots (s-k)(s-k-1) w_1\cdots w_{k}w_{k+1}}
  \\
  &\quad
  \cdot
  \sum_{\substack{b_{j}\in\{p_{j},q_{j}\}\\j=1,\dots,k,k+1}}
  (-1)^{\#\{b_{j}=q_{j}\}}
  (a+b_{1}w_1+\cdots+b_{k}w_{k}+b_{k+1}w_{k+1})^{k+1-s}
  \\
  &\quad
  +O\left(\sum_{\substack{b_{j}\in\{p_{j},q_{j}\}\\j=1,\dots,k,k+1}}
  (a+b_{1}+\cdots+b_{k}+b_{k+1})^{k-\sigma}
  \right)
\end{align*}
As in the case of $r=1$, the result follows from Lemma \ref{exp sum}.

In the case where exactly $d$ of $ q_1, \dots, q_r $ are infinite for some $d=1,\dots,r-1$.
First, letting $q_1 \rightarrow \infty$ in \eqref{sum(1/(a+mw))}, we have
\begin{align}
 \sum_{p_1\le m_1<\infty}\frac{1}{(a+m_1 w_1)^s}
 =\frac{(a+p_1 w_1)^{1-s}}{(s-1) w_1} 
  +O\left((a+p_1)^{-\sigma} \right) \label{sum(1/(a+mw)),infty}
\end{align}
for $\sigma>1$. 
The $r$-fold series on the left-hand side of \eqref{sum_{p,q}} includes $d$ infinite series among the $r$ series.
Since the series converges absolutely for $\sigma>d\ (1\le d\le r-1)$, the order of summation can be exchanged accordingly.
Therefore, we may relabel the indices 
$p_j, m_j, q_j$, and consider a series of the form
\begin{align*}
 &\sum_{p_{1}\le m_1 \le q_{1}}
 \cdots
 \sum_{p_{r}\le m_r \le q_{r}}
 \frac{1}{(a+\mathbf{m}\cdot\mathbf{w})^s}\\
 &=\sum_{p_{r}\le m_r \le q_{r}}\cdots \sum_{p_{d+1}\le m_{d+1} \le q_{d+1}}\ \sum_{p_{d}\le m_d < \infty} \cdots\sum_{p_{1}\le m_1 < \infty} \frac{1}{(a+\mathbf{m}\cdot\mathbf{w})^s}
\end{align*}
without loss of generality.
Using the results of \eqref{sum_{p,q}} and \eqref{sum(1/(a+mw))}, a similar computation can be performed.
In particular, by equation \eqref{sum(1/(a+mw)),infty}, since the contributions from the infinite sums do not involve any terms containing $q_j$, the desired formula in the lemma follows.
\end{proof}

\begin{lemma} \label{lemC}
 Let $s=\sigma+it \in\mathbb{C}$ with $\sigma>r-1$.
 Then the series
\begin{align*}
 \begin{split}
  \sum_{m_{1}>N}
   \cdots 
   \sum_{m_{r}>N}
   \frac{1}{(a+\mathbf{m}\cdot\mathbf{w})^s}
 \end{split}
 \end{align*}
 which converges absolutely for $\sigma>r$, 
admits an analytic continuation to the region $\sigma>r-1$, and satisfies the relation
 \begin{align*}
 \begin{split}
  \sum_{m_{1}>N}
   \cdots 
   \sum_{m_{r}>N}
   \frac{1}{(a+\mathbf{m}\cdot\mathbf{w})^s}
 =\frac{(a+w_{1}N+\cdots+w_{r}N)^{r-s}}{(s-1)\cdots (s-r)w_1\cdots w_{r} }
  %+ O(|s|^{r}N^{r-1-\sigma}),
  + O(|s|N^{r-1-\sigma}),
 \end{split}
 \end{align*}
 where the $O$-constants depend on $a$ and $w_1,\dots,w_r$.
\end{lemma}
\begin{proof}
 Let
\begin{align*}
 S_k
 =\sum_{m_{k}>N}
 \cdots 
 \sum_{m_{r}>N}
 \frac{1}{(a+\mathbf{m}\cdot\mathbf{w})^s}
 \qquad (k=1,\dots,r).
\end{align*}
To prove the lemma, we show
\begin{align*}
\begin{split}
 S_k
 &=\frac{(a+m_1w_1+\cdots +m_{r-1}w_{k-1}+w_{k}N+\cdots+w_{r}N)^{r-k+1-s}}{(s-1)\cdots (s-r+k-1)w_k\cdots w_{r}}\\
 &\quad 
  %+O(|s|^{r-k+1}(a+m_1w_1+\cdots+m_{k-1}w_{k-1}+w_{k}N+\cdots+w_rN)^{r-k-\sigma}),
  +O(|s|(a+m_1w_1+\cdots+m_{k-1}w_{k-1}+w_{k}N+\cdots+w_rN)^{r-k-\sigma}) \quad (\sigma>r-k)
 \end{split}
\end{align*}
for $k=1,\dots,r$ by backward induction on $k$. 
When $k=r$, by the Euler-Maclaurin summation formula, we have
\begin{align*}
 \sum_{a<n\le b}f(n)
 =\int_a^bf(\xi)d\xi+\int_a^b(\xi-[\xi])f'(\xi)d\xi+(a-[a])f(a)-(b-[b])f(b),
\end{align*}
where the symbol $[\xi]$ means the greatest integer less than or equal to $\xi$. 
Let $f(\xi)=(a+m_1w_1+\cdots+m_{r-1}w_{r-1}+w_r\xi)^{-s}$, and we take $a=N$ and $b\rightarrow \infty$.
For $\sigma>1$, we have
\begin{align*}
 S_r
 &=\sum_{m_r>N}\frac{1}{(a+m_1w_1+\cdots+m_rw_r)^s} \\
 &=\frac{(a+m_1w_1+\cdots+m_{r-1}w_{r-1}+w_rN)^{1-s}}{(s-1)w_r} \\
 & \quad - sw_r\int_{N}^{\infty}\frac{\xi-[\xi]}{(a+m_1w_1+\cdots+m_{r-1}w_{r-1}+w_r\xi)^{s+1}}d\xi \\
 &=\frac{(a+m_1w_1+\cdots+m_{r-1}w_{r-1}+w_rN)^{1-s}}{(s-1)w_r}\\
  &\quad 
   +O(|s|(a+m_1w_1+\cdots+m_{r-1}w_{r-1}+w_rN)^{-\sigma}).
\end{align*}
Moreover, this asymptotic expansion remains valid for $\sigma>0$.

Assuming the assumption holds at $k+1$ and $\sigma>r-k-1$, we have
\begin{align*}
 S_{k}
 &=\sum_{m_{k}>N}S_{k+1}\\
 &=
 \sum_{m_{k}>N}
 \sum_{m_{k+1}>N}
  \cdots
 \sum_{m_r>N}
 \frac{1}{(a+m_1w_1+\dots+m_kw_k+m_{k+1}w_{k+1}+\cdots+m_rw_r)^s}\\
 &=\sum_{m_{k}>N} \biggl\{
 \frac{(a+m_{1}w_{1}+\cdots +m_kw_k+w_{k+1}N+\cdots+w_{r}N)^{r-k-s}}{(s-1)\cdots (s-r+k)w_{k+1}\cdots w_{r}}\\
 &\quad \quad \quad \quad + O(|s|(a+m_1w_1+\cdots+m_{k}w_{k}+w_{k+1}N+\cdots+w_rN)^{r-k-1-\sigma})
 \biggr\}\\
 &=\frac{1}{(s-1)\cdots (s-r+k)w_{k+1}\cdots w_{r}}
 \sum_{m_{k}>N}(a+m_{1}w_{1}+\cdots +m_kw_k+w_{k+1}N+\cdots+w_{r}N)^{r-k-s}\\
 &\quad +O \left(|s|\sum_{m_k>N}(a+m_1w_1+\cdots+m_{k}w_{k}+w_{k+1}N+\cdots+w_rN)^{r-k-1-\sigma}\right).
\end{align*}
Here, the expression inside the $O$-term can be estimated as follows:
\begin{align*}
 &|s|\sum_{m_k>N}(a+m_1w_1+\cdots+m_{k}w_{k}+w_{k+1}N+\cdots+w_rN)^{r-k-1-\sigma} \\
 &=\frac{|s|(a+m_1w_1+\cdots+m_{k-1}w_{k-1}+w_{k}N+\cdots+w_rN)^{r-k-\sigma}}{(\sigma-r+k)w_k}  \\
 &\quad -|s|(\sigma-r+k)w_k\\
 &\qquad\cdot\int_{N}^{\infty}\frac{\xi-[\xi]}{(a+m_1w_1+\cdots+m_{k-1}w_{k-1}+w_{k}\xi+w_{k+1}N+\cdots+w_rN)^{\sigma-r+k+1}}d\xi\\
 &=\frac{|s|(a+m_1w_1+\cdots+m_{k-1}w_{k-1}+w_{k}N+\cdots+w_rN)^{r-k-\sigma}}{(\sigma-r+k)w_k}  \\
 &\quad +O(|s| (a+m_1w_1+\cdots+m_{k-1}w_{k-1}+w_{k}N+\cdots+w_rN)^{r-k-\sigma})\\
 &\ll 
 |s|(a+m_1w_1+\cdots+m_{k-1}w_{k-1}+w_{k}N+\cdots+w_rN)^{r-k-\sigma}.
 % + |s|(a+m_1w_1+\cdots+m_{k-1}w_{k-1}+w_{k}N+\cdots+w_rN)^{r-k-\sigma}
\end{align*}
Then we have
\begin{align*}
 S_{k}
 &=\frac{(a+m_{1}w_{1}+\cdots +m_{k-1}w_{k-1}+w_kN+\cdots+w_{r}N)^{r-k+1-s}}{(s-1)\cdots (s-r+k-1)w_k\cdots w_{r}}\\
 &\quad -\frac{w_k(s-r+k-1)}{(s-1)\cdots (s-r+k)w_{k+1}\cdots w_{r}}\\
 &\qquad\cdot\int_{N}^{\infty}
  \frac{\xi-[\xi]}{(a+m_{1}w_{1}+\cdots +m_{k-1}w_{k-1}+w_{k}\xi+w_{k+1}N+\cdots+w_{r}N)^{s-r+k+1}} d\xi\\
 %&\quad +O\left( |s|^{r-k}\sum_{m_k>N}(a+m_1w_1+\cdots+m_{k}w_{k}+w_{k+1}N+\cdots+w_rN)^{r-k-1-\sigma} \right) \\
 &\quad +O\left( |s|(a+m_1w_1+\cdots+m_{k-1}w_{k-1}+w_{k}N+\cdots+w_rN)^{r-k-\sigma}  \right) \\
 &=\frac{(a+m_{1}w_{1}+\cdots +m_{k-1}w_{k-1}+w_kN+\cdots+w_{r}N)^{r-k+1-s}}{(s-1)\cdots (s-r+k)w_k\cdots w_{r}}\\
 &\quad +O\left( |s|(a+m_1w_1+\cdots+m_{k-1}w_{k-1}+w_{k}N+\cdots+w_rN)^{r-k-\sigma}  \right)
\end{align*}
for $\sigma>r-k+1$.
Moreover, the above formula holds for $\sigma>r-k$.
Therefore, we have
\begin{align*}
 &\sum_{m_{1}>N}
   \cdots 
   \sum_{m_{r}>N}
   \frac{1}{(a+\mathbf{m}\cdot\mathbf{w})^s}\\
 &\quad =\frac{(a+w_{1}N+\cdots+w_{r}N)^{r-s}}{(s-1)\cdots (s-r)w_1\cdots w_{r}}
  %+ O(|s|^{r}N^{r-1-\sigma}),
  + O(|s|(a+w_1N+\cdots+w_rN)^{r-1-\sigma})
\end{align*}
for $\sigma>r-1$, so the proof of Lemma \ref{lemC} is complete. 
\end{proof}

\begin{lemma}\label{lemB}
For $r\ge 1$, a complex number $s$, and $1\le x\le N$, the following holds:
\begin{align*}
 &\sum_{\substack{(p_{j},q_{j})\in\{(0,x),(x,N)\}\\j=1,\dots,r\\((p_1,q_1),\dots,(p_r,q_r))\ne((0,x), \dots,(0,x))}}
 \sum_{\substack{b_{j}\in\{p_{j},q_{j}\}\\j=1,\dots,r}}
  (-1)^{\#\{b_{j}=q_{j}\}}
  \left(a+b_{1}w_1+\cdots+b_{r}w_{r}\right)^{r-s}\\
  %%%
  &=\sum_{\substack{E\subseteq\{w_1,\dots,w_r\}\\E\ne\emptyset}}
 (-1)^{\# E}
 \left\{
 \biggl(a+N\sum_{ e\in E }e\biggr)^{r-s}
 -\biggl(a+x\sum_{ e\in E }e\biggr)^{r-s}
 \right\}.
\end{align*}
\end{lemma}
\begin{proof}
We shall prove the case where $r=2,3$. 
The general case can be proved in the same way.
When $r=2$, we have
\begin{align*}
 &\textrm{L.H.S.}\\
 &=\sum_{\substack{(p_{1},q_{1})\in\{(0,x),(x,N)\}\\(p_{2},q_{2})\in\{(0,x),(x,N)\}\\((p_{1},q_{1}),(p_{2},q_{2}))\ne((0,x),(0,x))}}
 \sum_{b_{1}\in\{p_{1},q_{1}\}}
 \sum_{b_{2}\in\{p_{2},q_{2}\}}
  (-1)^{\#\{b_{j}=q_{j}|j=1,2\}}
  (a+b_{1}w_1+b_{2}w_{2})^{2-s}\\
 %%%
 &=\Biggl(
 \sum_{\substack{b_{1}\in\{p_{1}=x,q_{1}=N\}\\b_{2}\in\{p_{2}=0,q_{2}=x\}}}
 +\sum_{\substack{b_{1}\in\{p_{1}=0,q_{1}=x\}\\b_{2}\in\{p_{2}=x,q_{2}=N\}}}
 +\sum_{\substack{b_{1}\in\{p_{1}=x,q_{1}=N\}\\b_{2}\in\{p_{2}=x,q_{2}=N\}}}
 \Biggr)  
 (-1)^{\#\{b_{j}=q_{j}|j=1,2\}}
 (a+b_{1}w_1+b_{2}w_{2})^{2-s}\\
 %%%
 &=(a+w_{1}x)^{2-s}
 -(a+w_{1}x+w_{2}x)^{2-s}
 -(a+w_1N)^{2-s}
 +(a+w_1N+w_{2}x)^{2-s}\\
 &\quad
 +(a+w_{2}x)^{2-s}
 -(a+w_{2}N)^{2-s}
 -(a+w_1x+w_{2}x)^{2-s}
 +(a+w_1x+w_{2}N)^{2-s}\\
 &\quad
 +(a+w_1x+w_{2}x)^{2-s}
 -(a+w_1x+w_{2}N)^{2-s}
 -(a+w_1N+w_{2}x)^{2-s}
 +(a+w_1N+w_{2}N)^{2-s}\\
 %%%
 &=(a+w_{1}x)^{2-s}
 +(a+w_{2}x)^{2-s}
 -(a+w_{1}x+w_{2}x)^{2-s}\\
 &\quad
 -(a+w_1N)^{2-s} 
 -(a+w_{2}N)^{2-s}
 +(a+w_1N+w_{2}N)^{2-s}\\
 &=\textrm{R.H.S.}
\end{align*}

When $r=3$, we also have
\begin{align*}
 &\textrm{L.H.S.}\\
 &=\sum_{\substack{(p_{1},q_{1})\in\{(0,x),(x,N)\}\\(p_{2},q_{2})\in\{(0,x),(x,N)\}\\(p_{3},q_{3})\in\{(0,x),(x,N)\}\\((p_{1},q_{1}),(p_{2},q_{2}),(p_{3},q_{3}))\ne((0,x),(0,x),(0,x))}}
 \sum_{b_{1}\in\{p_{1},q_{1}\}}
 \sum_{b_{2}\in\{p_{2},q_{2}\}}
 \sum_{b_{3}\in\{p_{3},q_{3}\}}\\
 &\qquad\qquad
  (-1)^{\#\{b_{j}=q_{j}|j=1,2,3\}}
  (a+b_{1}w_1+b_{2}w_{2}+b_{3}w_3)^{3-s}\\
 %%%
 &=\Biggl(
 \sum_{\substack{b_{1}\in\{p_{1}=x,q_{1}=N\}\\b_{2}\in\{p_{2}=0,q_{2}=x\}\\b_{3}\in\{p_{3}=0,q_{3}=x\}}}
 +\sum_{\substack{b_{1}\in\{p_{1}=0,q_{1}=x\}\\b_{2}\in\{p_{2}=x,q_{2}=N\}\\b_{3}\in\{p_{3}=0,q_{3}=x\}}}
 +\sum_{\substack{b_{1}\in\{p_{1}=0,q_{1}=x\}\\b_{2}\in\{p_{2}=0,q_{2}=x\}\\b_{3}\in\{p_{3}=x,q_{3}=N\}}}\\
 &\qquad
 +\sum_{\substack{b_{1}\in\{p_{1}=x,q_{1}=N\}\\b_{2}\in\{p_{2}=x,q_{2}=N\}\\b_{3}\in\{p_{3}=0,q_{3}=x\}}}
 +\sum_{\substack{b_{1}\in\{p_{1}=x,q_{1}=N\}\\b_{2}\in\{p_{2}=0,q_{2}=x\}\\b_{3}\in\{p_{3}=x,q_{3}=N\}}}
 +\sum_{\substack{b_{1}\in\{p_{1}=0,q_{1}=x\}\\b_{2}\in\{p_{2}=x,q_{2}=N\}\\b_{3}\in\{p_{3}=x,q_{3}=N\}}}
 +\sum_{\substack{b_{1}\in\{p_{1}=x,q_{1}=N\}\\b_{2}\in\{p_{2}=x,q_{2}=N\}\\b_{3}\in\{p_{3}=x,q_{3}=N\}}}
 \Biggr)\\ 
 &\qquad\qquad
  (-1)^{\#\{b_{j}=q_{j}|j=1,2,3\}}
  (a+b_{1}w_1+b_{2}w_{2}+b_{3}w_3)^{3-s}.
\end{align*}
By carefully observing the signs of each term and considering the terms that cancel out, we obtain the right-hand side.
Indeed, after grouping and canceling terms appropriately, we get
\begin{align*}
 &\textrm{L.H.S.}\\
 &=(a+w_{1}x)^{3-s}
 +(a+w_{2}x)^{3-s}
 +(a+w_{3}x)^{3-s}
 -(a+w_{1}x+w_{2}x)^{3-s}\\
 &\quad 
 -(a+w_{1}x+w_{3}x)^{3-s}
 -(a+w_{2}x+w_{3}x)^{3-s}
 +(a+w_{1}x+w_{2}x+w_{3}x)^{3-s}\\
  &\quad
 -(a+w_{1}N)^{3-s}
 -(a+w_{2}N)^{3-s}
 -(a+w_{3}N)^{3-s}
 +(a+w_{1}N+w_{2}N)^{3-s}\\
 &\quad 
 +(a+w_{1}N+w_{3}N)^{3-s}
 +(a+w_{2}N+w_{3}N)^{3-s}
 -(a+w_{1}N+w_{2}N+w_{3}N)^{3-s}\\
 %%%
 &=\textrm{R.H.S.}
\end{align*}
as claimed.
\end{proof}

\begin{proof}[Proof of Theorem \ref{main2}]
For $\sigma>r$, by definition we decompose
\begin{align*}
 \zeta_r(s,a,\mathbf{w})
 %%%
 &=\sum_{0\le m_1 \le N}
 \cdots 
 \sum_{0\le m_r \le N}
 \frac{1}{(a+\mathbf{m}\cdot\mathbf{w})^s}
 +\sum_{\substack{S\subsetneq\{1,\dots,r\}\\S\neq\emptyset}} 
 \sum_{\mathbf{m}\in Q_S}
 \frac{1}{(a+\mathbf{m}\cdot\mathbf{w})^s}\\
 &\quad +\sum_{m_1>N}
 \cdots
 \sum_{m_r>N}
 \frac{1}{(a+\mathbf{m}\cdot\mathbf{w})^s},
\end{align*}
where 
$
 Q_S
 \coloneqq\left\{ \mathbf{m} \mid m_l > N \text{ for } l \in S, \; m_j \le N \text{ for } j \notin S \right\}.
$
%%%
The second term of the above equation corresponds to all combinations in Lemma 2 where each 
$ (p_j,q_j)\ (j=1,\dots,r)$ is taken from either $(0,N)$ or $(N,\infty)$, 
except when all pairs are chosen from the same side.
Then we have
\begin{align}
\begin{split}
 &\sum_{\substack{S\subsetneq\{1,\dots,r\}\\S\neq\emptyset}} 
 \sum_{\mathbf{m}\in Q_S}
 \frac{1}{(a+\mathbf{m}\cdot\mathbf{w})^s}\\
 %%%  
 &=-\sum_{\substack{E\subsetneq\{w_1,\dots,w_r\}\\E\ne\emptyset}}
 (-1)^{\#E}
 \frac{(a+N\sum_{ e\in E }e)^{r-s}}{(s-1)\cdots (s-r) w_1\cdots w_{r}}
 +O(N^{r-1-\sigma}) 
\end{split}
\end{align}
by Lemma \ref{lemA}. Moreover, 
\begin{align*}
 \begin{split}
  \sum_{m_{1}>N}
   \cdots 
   \sum_{m_{r}>N}
   \frac{1}{(a+\mathbf{m}\cdot\mathbf{w})^s}
 =\frac{(a+w_{1}N+\cdots+w_{r}N)^{r-s}}{(s-1)\cdots (s-r)w_1\cdots w_{r} }
  + O(|s|N^{r-1-\sigma}).
 \end{split}
\end{align*}
by Lemma \ref{lemC}. Thus we have
\begin{align*}
 \zeta_r(s,a,\mathbf{w})
 &=
 \sum_{0\le m_1 \le N}
 \cdots 
 \sum_{0\le m_r \le N}
 \frac{1}{(a+\mathbf{m}\cdot\mathbf{w})^s}\\
 &\quad
 -\sum_{\substack{E\subseteq\{w_1,\dots,w_r\}\\E\ne\emptyset}}
 (-1)^{\#E}
 \frac{(a+N\sum_{ e\in E }e)^{r-s}}{(s-1)\cdots (s-r) w_1\cdots w_{r}}
 +O(|s|N^{r-1-\sigma})
\end{align*}
for $\sigma>r-1$. On the other hand, we decompose
\begin{align*}
 &\sum_{0\le m_1 \le N}
 \cdots 
 \sum_{0\le m_r \le N}
 \frac{1}{(a+\mathbf{m}\cdot\mathbf{w})^s}\\
 %%%
 &=
  \sum_{0\le m_1 \le x}
  \cdots 
  \sum_{0\le m_r \le x}
  \frac{1}{(a+\mathbf{m}\cdot\mathbf{w})^s}
   +\sum_{\substack{S\subseteq\{1,\dots,r\}\\S\neq\emptyset}} 
   \sum_{\mathbf{m}\in R_S}
   \frac{1}{(a+\mathbf{m}\cdot\mathbf{w})^s},
\end{align*}
where 
$
 R_S
 \coloneqq\left\{ \mathbf{m} \mid x<m_l \le N\text{ for } l \in S, \; 0\le m_j \le x \text{ for } j \notin S \right\}
$.
By Lemmas \ref{lemA} and \ref{lemB}, we have
\begin{align*}
 &\sum_{0\le m_1 \le N}
 \cdots 
 \sum_{0\le m_r \le N}
 \frac{1}{(a+\mathbf{m}\cdot\mathbf{w})^s}\\
 %%%
 &=\sum_{0\le m_1 \le x}
 \cdots 
 \sum_{0\le m_r \le x}
 \frac{1}{(a+\mathbf{m}\cdot\mathbf{w})^s}\\
 %%%
 &\quad
 +\sum_{\substack{E\subseteq\{w_1,\dots,w_r\}\\E\ne\emptyset}}
 (-1)^{\# E}
 \frac{(a+N\sum_{ e\in E }e)^{r-s}-(a+x\sum_{ e\in E }e)^{r-s}}{(s-1)\cdots(s-r)w_1\cdots w_r}
 +O(x^{r-\sigma-1})+O(|s|N^{r-\sigma-1}). 
\end{align*}
Thus we find
\begin{align*}
 \zeta_r(s,a,\mathbf{w})
 %%%
 &=\sum_{0\le m_1 \le x}
 \cdots 
 \sum_{0\le m_r \le x}
 \frac{1}{(a+\mathbf{m}\cdot\mathbf{w})^s}\\
 &\quad
  -\sum_{\substack{E\subseteq\{w_1,\dots,w_r\}\\E\ne\emptyset}}
 (-1)^{\# E}
 \frac{(a+x\sum_{ e\in E }e)^{r-s}}{(s-1)\cdots(s-r)w_1\cdots w_r}
 +O(x^{r-\sigma-1})+O(|s|N^{r-\sigma-1}).
\end{align*}
Taking the limit as $N\to\infty$, we obtain the desired conclusion.
\end{proof}

%%%%%%%%%%%%%%%%%%%%%%%%%%%%
\section{Proof of Theorem \ref{main3}} 
\begin{proof}[Proof of Theorem \ref{main3}]
Setting $ C =2\pi$ and $x=t$ in Theorem \ref{main2}, we have
\begin{align}\label{sigma+O}
 \zeta_r(s,a,\mathbf{w})
 &=\sum_{0\le m_1\le t}
  \cdots 
  \sum_{0\le m_r\le t}
  \frac{1}{(a+\mathbf{m}\cdot \mathbf{w})^{s}}
  +O(t^{r-1-\sigma}).
\end{align}
Let the first term on the right-hand side be $\Sigma(s)$. 
Then we derive
\begin{align*}
 \int_1^T |\Sigma(s)|^2 dt
 =\int_1^T 
  \Biggl(
  \sum_{m_1 \le t,\dots, m_r \le t} 
  \frac{1}{(a+\mathbf{m}\cdot \mathbf{w})^{\sigma+it}}
  \sum_{n_{1} \le t,\dots, n_{r} \le t} 
  \frac1{(a+\mathbf{n}\cdot \mathbf{w})^{\sigma-it}} 
  \Biggr)
  \,dt.
\end{align*}
Now we change the order of summation and integration. 
From the conditions $m_k\le t,n_k\le t$ for $k=1,\dots,r$, we deduce that the range of $t$ is given by
$M\coloneqq\max_{1\le k\le r}\{m_k,n_k\}\le t\le T$. 
Using this, we obtain
\begin{align*}
 \int_1^T |\Sigma(s)|^2 dt
 &=\sum_{m_1 \le T,\dots, m_r \le T} 
   \frac{1}{(a+\mathbf{m}\cdot \mathbf{w})^\sigma}
  \sum_{n_{1} \le T,\dots, n_{r} \le T} 
   \frac{1}{(a+\mathbf{n}\cdot \mathbf{w})^\sigma}
  \int_M^T \left( \frac{a+\mathbf{m}\cdot \mathbf{w}}{a+\mathbf{n}\cdot \mathbf{w}} \right)^{it} dt\\
 &= \sum_{\substack{0\le m_1,\dots m_r \le T \\ 0\le n_1,\dots n_r \le T \\ \mathbf{m}\cdot \mathbf{w}=\mathbf{n}\cdot \mathbf{w} } } \frac{1}{(a+\mathbf{m}\cdot \mathbf{w})^\sigma (a+\mathbf{n}\cdot \mathbf{w})^\sigma}
 \cdot (T -M) \\
 &\quad+\sum_{\substack{0\le m_1,\dots m_r \le T \\ 0\le n_1,\dots n_r \le T \\ \mathbf{m}\cdot \mathbf{w} \neq \mathbf{n}\cdot \mathbf{w} } } \frac{1}{(a+\mathbf{m}\cdot \mathbf{w})^\sigma (a+\mathbf{n}\cdot \mathbf{w})^\sigma}   \\
 &\qquad \cdot \frac{e^{iT \log \{(a+\mathbf{m}\cdot \mathbf{w})/(a+\mathbf{n}\cdot \mathbf{w}) \}}
 -e^{iM \log \{(a+\mathbf{m}\cdot \mathbf{w})/(a+\mathbf{n}\cdot \mathbf{w}) \}}
 }
 {i \log \{(a+\mathbf{m}\cdot \mathbf{w})/(a+\mathbf{n}\cdot \mathbf{w}) \}}.
\end{align*}

Let $S_1$ be the first term on the right-hand side of the above equality and $S_2$ be the second term, respectively. 
We first evaluate $S_1$ and obtain the following:
\begin{align}
 S_1
 =T
 (\tilde{\zeta_{r}}(\sigma,a,\mathbf{w})-U)
 -
 \sum_{\substack{
  0\le m_1,\dots,m_r\le T\\
  0\le n_{1},\dots,n_{r}\le T\\
  \mathbf{m}\cdot\mathbf{w}=\mathbf{n}\cdot\mathbf{w}
  }}
  \frac{M}{(a+\mathbf{m}\cdot\mathbf{w})^{\sigma}(a+\mathbf{n}\cdot\mathbf{w})^{\sigma}},
 \label{S_1}
\end{align}
where
\[
 U
 \coloneqq
 \sum_{\substack{
  m_1,\dots,m_r\ge0\\
  n_{1},\dots,n_{r}\ge0\\
  \max_{1\le k\le r}\{m_k,n_k\} > T\\  \mathbf{m}\cdot\mathbf{w}=\mathbf{n}\cdot\mathbf{w}
  }}
  \frac{1}{
   (a+\mathbf{m}\cdot \mathbf{w})^{\sigma}
   (a+\mathbf{n}\cdot \mathbf{w})^{\sigma}
  }.
\]
%
% Note that the conditions that $M>T$ and $\min_{1\le k\le r}\{m_k, n_k\}\le T$ mean that at least one of $m_1,\dots,m_r,n_1,\dots,n_r$ is less than $T$, while at least one is greater than or equal to $T$.  
%
We decompose $U=U_1+U_2$, where
\[
U_1:=
\sum_{\substack{
m_1,\dots,m_r\ge0\\
n_1,\dots,n_r\ge0\\
\max_{1\le k\le r}\{m_k,n_k\}>T\\
\min_{1\le k\le r}\{m_k,n_k\}\le T\\
\mathbf m\cdot\mathbf w=\mathbf n\cdot\mathbf w
}}
\frac{1}{(a+\mathbf m\cdot\mathbf w)^\sigma (a+\mathbf n\cdot\mathbf w)^\sigma},
\]
and
\[
U_2:=
\sum_{\substack{
m_1,\dots,m_r>T\\
n_1,\dots,n_r>T\\
\mathbf m\cdot\mathbf w=\mathbf n\cdot\mathbf w
}}
\frac{1}{(a+\mathbf m\cdot\mathbf w)^\sigma (a+\mathbf n\cdot\mathbf w)^\sigma}.
\]
We first estimate $U_1$.
Since $a+\mathbf{m}\cdot \mathbf{w}\asymp 1+m_1+\cdots+m_r$, it follows that
\[
 U_1
 \ll
 \sum_{\substack{
  m_1,\dots,m_r\ge0\\
  n_{1},\dots,n_{r}\ge0\\
  \max_{1\le k\le r}\{m_k,n_k\} > T\\
  \min_{1\le k\le r}\{m_k,n_k\}\le T\\
  \mathbf{m}\cdot\mathbf{w}=\mathbf{n}\cdot\mathbf{w}
  %m_1+\cdots+m_r\asymp n_1+\cdots+n_r
  }}
  \frac{1}{(1+m_1+\cdots+m_r)^{\sigma}(1+n_1+\cdots+n_r)^{\sigma}}.
\]
Thus, we find that
\begin{align*}
 U_1
 &\ll
  \sum_{\substack{
  0\le d_1,d_2\le r\\
  \max\{d_1,d_2\}>0\\
  \min\{d_1,d_2\}<r
  }}
  \sum_{\substack{
  m_1,\dots,m_r\ge0\\
  n_{1},\dots,n_{r}\ge0\\
  \#\{m_1,\dots,m_r> T\}=d_1\\
  \#\{n_1,\dots,n_r> T\}=d_2\\
  \mathbf{m}\cdot\mathbf{w}=\mathbf{n}\cdot\mathbf{w}
  }}
  \frac{1}{(1+m_1+\cdots+m_r)^{\sigma}(1+n_1+\cdots+n_r)^{\sigma}}.
\end{align*}
If we fix $\mathbf{m}\cdot\mathbf{w}$ on the second sum and consider the other conditions, the number of running indices is $2r-2$ i.e., $m_1,\dots,m_{r-1}$ and $n_1,\dots,n_{r-1}$ for the following two reasons. 
First, when the value of $\mathbf{m}\cdot\mathbf{w}$ is given, $m_r$ is uniquely determined if $m_1, \dots, m_{r-1}$ and $\mathbf{w}$ are fixed. 
Second, when $\mathbf{m}\cdot\mathbf{w}=\mathbf{n}\cdot\mathbf{w}$ holds, $n_r$ is uniquely determined if $\mathbf{m}$, $n_1, \dots, n_{r-1}$, and $\mathbf{w}$ are fixed.
Since $m_1+\dots+m_{r-1}\asymp n_1+\dots+n_{r-1}$ when $\mathbf{m}\cdot\mathbf{w}=\mathbf{n}\cdot\mathbf{w}$, we have
\begin{align*}
 U_1
 %%% 
 &\ll 
  \sum_{\substack{
  0\le d_1,d_2\le r\\
  \max\{d_1,d_2\}>0\\
  \min\{d_1,d_2\}<r%\\
  %\mathbf{m}\cdot\mathbf{w}=\mathbf{n}\cdot\mathbf{w}
  }}
  \sum_{m_1+\dots+m_{r-1}\asymp n_1+\dots+n_{r-1}}\\
  &\quad
  \sum_{m_1>T}\;\cdots\sum_{m_{d_1}>T}\;
  \sum_{0\le m_{d_1+1}\le T}\;\cdots\sum_{0\le m_{r-1}\le T}\;
  \sum_{n_1>T}\;\cdots\sum_{n_{d_2}>T}\;
  \sum_{0\le n_{d_2+1}\le T}\;\cdots\sum_{0\le n_{r-1}\le T}\;\\
  &\quad 
  \frac{1}{(1+m_1+\cdots+m_{r-1})^{\sigma}(1+n_1+\cdots+n_{r-1})^{\sigma}}\\
 %%%
 &\ll 
  \sum_{\substack{
  0\le d_1,d_2\le r\\
  \max\{d_1,d_2\}>0\\
  \min\{d_1,d_2\}<r%\\
  }}
  \sum_{m_1+\dots+m_{d_1}\asymp n_1+\dots+n_{d_2}}\\
  &\quad
  \sum_{m_1>T}\;\cdots\sum_{m_{d_1}>T}\;
  \sum_{n_1>T}\;\cdots\sum_{n_{d_2}>T}
  \frac{ T^{2r-2-(d_1+d_2)} }{(1+m_1+\cdots+m_{d_1})^{\sigma}(1+n_1+\cdots+n_{d_2})^{\sigma}}\\
 %%%
 &\ll 
  \sum_{\substack{
  0\le d_1,d_2\le r\\
  \max\{d_1,d_2\}>0\\
  \min\{d_1,d_2\}<r
  }}
  T^{2r-2-(d_1+d_2)}
  \sum_{m_1+\dots+m_{d_1}\asymp n_1+\dots+n_{d_2}}\\
  &\quad
  \cdot\left(
  \sum_{m_1>T}\;\cdots\sum_{m_{d_1}>T}\;
  \frac{ 1 }{(1+m_1+\cdots+m_{d_1})^{\sigma}}
  \right)
  \cdot\left(
  \sum_{n_1>T}\;\cdots\sum_{n_{d_2}>T}
  \frac{ 1 }{(1+n_1+\cdots+n_{d_2})^{\sigma}}
  \right).
\end{align*}
Here, 
\begin{align*}
 \sum_{m_1>T}\;\cdots\sum_{m_{d_1}>T}\;
 \frac{1}{(1+m_1+\cdots+m_{d_1})^{\sigma}}
 &\ll 
 \sum_{j> d_1 T} \frac{1}{(1+j)^{\sigma}}
 \cdot \binom{j}{d_1-1}\\
 &\ll 
 \sum_{j> d_1 T} j^{-\sigma+d_1-1}\\ 
 &\ll 
 T^{-\sigma+d_1}
\end{align*}
if $d_1\ge 1$, 
\begin{align*}
 \sum_{n_1>T}\;\cdots\sum_{n_{d_2}>T}\;
 \frac{1}{(1+n_1+\cdots+n_{d_2})^{\sigma}}
 &\ll 
 T^{-\sigma+d_2}.
\end{align*}
%if $d_2\ge 1$, and at least one of $d_1$ and $d_2$ is greater than or equal to $2$.
%
By considering the sum $\sum_{m_1+\dots+m_{d_1}\asymp n_1+\dots+n_{d_2}}$, we have
\begin{align*}
 U_1
 %%% 
 \ll 
 T^{2r-2\sigma-1}.
\end{align*}
For $U_2$, all the variables $m_1,\dots,m_r,n_1,\dots,n_r$ exceed $T$, and arguing in the same way as above with $d_1=d_2=r$, we obtain
$
U_2\ll T^{2r-2\sigma-1} $.
Hence
\[
U\ll T^{2r-2\sigma-1}.
\]
The sum involving $M$ in \eqref{S_1} is estimated, since $M\ll m_1+\cdots+m_r+n_1+\cdots+n_r $ in this case, as
\begin{align*}
 &\sum_{\substack{
 0\le m_1,\dots,m_r\le T\\
 0\le n_{1},\dots,n_{r}\le T\\
 \mathbf{m}\cdot\mathbf{w}=\mathbf{n}\cdot\mathbf{w}
 }}
 \frac{M}{(a+\mathbf{m}\cdot\mathbf{w})^{\sigma}(a+\mathbf{n}\cdot\mathbf{w})^{\sigma}} \\
 &\ll
\sum_{\substack{
 0\le m_1,\dots,m_r\le T\\
 0\le n_{1},\dots,n_{r}\le T\\
 m_1+\cdots+m_r \asymp n_1+\cdots+ n_r
 }}
 \frac{ m_1+\cdots+m_r+n_1+\cdots+n_r}{(1+m_1+\cdots+m_r)^{\sigma}(1+ n_1+\cdots+ n_r)^{\sigma}} \\
 &\ll
 \sum_{0\le j\le 2rT}
 \sum_{\substack{
 0\le m_1,\dots,m_{r-1}\le T\\
 0\le n_{1},\dots,n_{r-1}\le T\\
 }}
 \frac{1}{(1+j)^{2\sigma-1}} \\
 &\ll
 \sum_{0\le j\le 2rT}
 \frac{j^{2r-2}}{(1+j)^{2\sigma-1}} \\
 &\ll
 \sum_{0\le j\le 2rT}
 \frac{1}{(1+j)^{2\sigma-2r +1}}    \\
 &\ll
  \begin{cases}
 %1, & \sigma>r,\\
 T^{2r-2\sigma}
 &\text{if } r-1\le\sigma<r,\\
 \log{T}
 &\text{if }\sigma=r.
 \end{cases}
\end{align*}
Then we have
\begin{align*}
 S_1
 =
 \tilde{\zeta_{r}}(\sigma,a,\mathbf{w})T
+
  \begin{cases}
 %1, & \sigma>r,\\
 O(T^{2r-2\sigma})
 &\text{if } r-1\le\sigma<r,\\
 O(\log{T})
 &\text{if }\sigma=r.
 \end{cases}
\end{align*}
Regarding $S_2$, we find
\begin{align*}
 S_2
 &\ll
  \sum_{\substack{
   0\le m_1,\dots,m_r\le T\\
   0\le n_{1},\dots,n_{r}\le T\\
   \mathbf{m}\cdot\mathbf{w}<\mathbf{n}\cdot\mathbf{w}
   }}
   \frac{1}{(a+\mathbf{m}\cdot\mathbf{w})^{\sigma}(a+\mathbf{n}\cdot\mathbf{w})^{\sigma}}
   \cdot
   \frac{1}{\log\{(a+\mathbf{n}\cdot\mathbf{w})/(a+\mathbf{m}\cdot\mathbf{w})\}}\\
 %%%  
 &=
  \left(
   \sum_{\substack{
   0\le m_1,\dots,m_r\le T\\
   0\le n_{1},\dots,n_{r}\le T\\
   a+\mathbf{m}\cdot\mathbf{w}
   <a+\mathbf{n}\cdot\mathbf{w}
   <2(a+\mathbf{m}\cdot\mathbf{w})
   }}
   +\sum_{\substack{
   0\le m_1,\dots,m_r\le T\\
   0\le n_{1},\dots,n_{r}\le T\\
   2(a+\mathbf{m}\cdot\mathbf{w})\le a+\mathbf{n}\cdot\mathbf{w}
   }}
  \right)\\
   &\qquad\qquad \frac{1}{(a+\mathbf{m}\cdot\mathbf{w})^{\sigma}(a+\mathbf{n}\cdot\mathbf{w})^{\sigma}}
   \cdot
   \frac{1}{\log\{(a+\mathbf{n}\cdot\mathbf{w})/(a+\mathbf{m}\cdot\mathbf{w})\}}.
\end{align*}
Let $W_1$ be the first term on the right-hand side of the above equality and $W_2$ be the second term. 
We have
\begin{align*}
 W_2
 &\ll
  \sum_{\substack{
   0\le m_1,\dots,m_r\le T\\
   0\le n_{1},\dots,n_{r}\le T\\
   2(a+\mathbf{m}\cdot\mathbf{w})\le a+\mathbf{n}\cdot\mathbf{w}
   }}
   \frac{1}{(a+\mathbf{m}\cdot\mathbf{w})^{\sigma}(a+\mathbf{n}\cdot\mathbf{w})^{\sigma}}\\
 %%%
 &\ll 
  \left(
   \sum_{0\le m_1,\dots,m_r\le T}
   \frac{1}{(1+m_1+\dots+m_r)^{\sigma}}
  \right)^2.
\end{align*}
Let $n=m_1+\dots+m_r$. 
The number of solutions to $m_1+\dots+m_r=n$ with $0\le m_i\le T$ is bounded by
\[
 \#\{(m_1,\dots,m_r)\mid m_1+\dots+m_r=n\} \ll n^{r-1}.
\]
This comes from the combinatorial fact that the number of solutions to a sum of $r$ nonnegative integers is bounded by $\binom{n+r-1}{r-1}$, which is $n^{r-1}$.
Then the sum becomes
\[
 \sum_{0\le m_1, m_2,\dots, m_r\le T} 
 \frac{1}{(1+m_1+\dots+m_r)^\sigma} 
 \ll \sum_{n=0}^{rT} \frac{n^{r-1}}{(1+n)^\sigma}
 \ll \int_{1}^{rT} x^{r-1-\sigma} \,dx.
\]
%
% Since
% \[
% \frac{n^{r-1}}{(1+n)^\sigma} 
% \ll n^{r-1-\sigma},
% \]
Thus we obtain
\[
 W_2 
 \ll
 \begin{cases}
 %1, & \sigma>r,\\
 T^{2r-2\sigma}
 &\text{if }r-1<\sigma<r,\\
 (\log T)^2
 &\text{if }\sigma=r.
 \end{cases}
\]
We also have the following evaluation of $W_1$:
\begin{align*}
 W_1
 &\ll
   \sum_{\substack{
   0\le m_1,\dots,m_r\le T\\
   0\le n_{1},\dots,n_{r}\le T\\
   a+\mathbf{m}\cdot\mathbf{w}
   <a+\mathbf{n}\cdot\mathbf{w}
   <2(a+\mathbf{m}\cdot\mathbf{w})
   }}
   \frac{1}{(a+\mathbf{m}\cdot\mathbf{w})^{\sigma}(a+\mathbf{n}\cdot\mathbf{w})^{\sigma}}
   \cdot
   \frac{1}{\log\{(a+\mathbf{n}\cdot\mathbf{w})/(a+\mathbf{m}\cdot\mathbf{w})\}}.
\end{align*}
We fix $m_1,\dots,m_r$ and $n_1,\dots,n_{r-1}$.
Then the condition
\[
a+\mathbf{m}\cdot\mathbf{w}
<
a+\mathbf{n}\cdot\mathbf{w}
<
2(a+\mathbf{m}\cdot\mathbf{w})
\]
is equivalent to
\[
\frac{1}{w_r}\sum_{j=1}^{r-1}(m_j-n_j)w_j+m_r
<
n_r
<
\frac{a}{w_r}
+\frac{1}{w_r}\sum_{j=1}^{r-1}(2m_j-n_j)w_j
+2m_r .
\]
Let $\varepsilon=\varepsilon(m_1,\dots,m_r,n_1,\dots,n_{r-1})$ and
$\delta=\delta(a,m_1,\dots,m_r,n_1,\dots,n_{r-1})$ be such that
$0\le \varepsilon,\delta<1$ and
\begin{align*}
 \frac{1}{w_r}\sum_{j=1}^{r-1}(m_j-n_j)w_j+m_r+\varepsilon &\in \mathbb Z,\\
 \frac{a}{w_r}
 +\frac{1}{w_r}\sum_{j=1}^{r-1}(2m_j-n_j)w_j
 +2m_r-\delta &\in \mathbb Z.
\end{align*}
Then
\[
K:=
\frac{a}{w_r}
+\frac{1}{w_r}\sum_{j=1}^{r-1}m_jw_j
+m_r-\varepsilon-\delta
\]
is a nonnegative integer, and every admissible $n_r$ can be written as
\[
n_r=
\frac{1}{w_r}\sum_{j=1}^{r-1}(m_j-n_j)w_j+m_r+\varepsilon+k
\]
for some $k=0,1,\dots,K$.
Hence
\[
a+\mathbf{n}\cdot\mathbf{w}
=
a+\mathbf{m}\cdot\mathbf{w}+w_r(\varepsilon+k).
\]
Since the case $\varepsilon=k=0$ would imply
$a+\mathbf{n}\cdot\mathbf{w}=a+\mathbf{m}\cdot\mathbf{w}$, which is excluded in $W_1$,
we may restrict the inner sum to those $k$ with $\varepsilon+k>0$.
Moreover, because
\[
0<\frac{w_r(\varepsilon+k)}{a+\mathbf{m}\cdot\mathbf{w}}<1
\]
in the range of $W_1$, we have
\[
\log\left(1+\frac{w_r(\varepsilon+k)}{a+\mathbf{m}\cdot\mathbf{w}}\right)
\asymp
\frac{w_r(\varepsilon+k)}{a+\mathbf{m}\cdot\mathbf{w}}.
\]
Since $d=1$, there exist $\lambda>0$ and positive integers
$p_1,\dots,p_r$ such that $w_j=\lambda p_j$ $(1\le j\le r)$.
Hence
\[
 \frac{1}{w_r}\sum_{j=1}^{r-1}(m_j-n_j)w_j+m_r
 =
 \frac{1}{p_r}\sum_{j=1}^{r-1}(m_j-n_j)p_j+m_r,
\]
so $\varepsilon$ is either $0$ or satisfies $\varepsilon\ge 1/p_r$.
Moreover, since $\mathbf m\cdot\mathbf w\ne \mathbf n\cdot\mathbf w$, we have
$k+\varepsilon>0$ and 
\[
 \sum_{\substack{0\le k\le K\\ k+\varepsilon>0}} \frac{1}{k+\varepsilon}
 \ll \log(2+K).
\]
Therefore,
\begin{align*}
 W_1
 &\ll
 \sum_{0\le m_1,\dots,m_r\le T}
 \sum_{\substack{0\le n_1,\dots,n_{r-1}\le T\\
 a+\sum_{j=1}^{r-1}n_jw_j<2(a+\mathbf{m}\cdot\mathbf{w})}}
 \sum_{\substack{0\le k\le K\\ \varepsilon+k>0}}
 \frac{1}{(a+\mathbf{m}\cdot\mathbf{w})^{\sigma}
 (a+\mathbf{m}\cdot\mathbf{w}+w_r(\varepsilon+k))^{\sigma}} \\
 &\qquad\qquad\qquad\qquad\qquad \cdot
 \frac{a+\mathbf{m}\cdot\mathbf{w}}{w_r(\varepsilon+k)}.
\end{align*}
Now, since all $w_j$ are fixed positive real numbers and
\[
a+\sum_{j=1}^{r-1}n_jw_j<2(a+\mathbf{m}\cdot\mathbf{w}),
\]
each $n_j$ $(1\le j\le r-1)$ is bounded by $O(a+\mathbf{m}\cdot\mathbf{w})$.
Also,
\[
K\asymp a+\mathbf{m}\cdot\mathbf{w}\asymp 1+m_1+\cdots+m_r.
\]
The number of possible $(n_1,\dots,n_{r-1})$ is
\[
\ll (a+\mathbf{m}\cdot\mathbf{w})^{r-1},
\]
and since $d=1$, there exist $\lambda>0$ and positive integers
$p_1,\dots,p_r$ such that $w_j=\lambda p_j$ $(1\le j\le r)$.
Hence
\[
 \frac{1}{w_r}\sum_{j=1}^{r-1}(m_j-n_j)w_j+m_r
 =
 \frac{1}{p_r}\sum_{j=1}^{r-1}(m_j-n_j)p_j+m_r,
\]
so $\varepsilon$ is either $0$ or satisfies $\varepsilon\ge 1/p_r$.
Moreover, since $\mathbf m\cdot\mathbf w\ne \mathbf n\cdot\mathbf w$, we have
$k+\varepsilon>0$. Therefore,
\[
\sum_{\substack{0\le k\le K\\ \varepsilon+k>0}} \frac{1}{\varepsilon+k}
\ll \log(K+2).
\]
Using also
\[
a+\mathbf{m}\cdot\mathbf{w}+w_r(\varepsilon+k)\asymp a+\mathbf{m}\cdot\mathbf{w},
\]
we obtain
\begin{align*}
 W_1
 &\ll
 \sum_{0\le m_1,\dots,m_r\le T}
 \frac{(a+\mathbf{m}\cdot\mathbf{w})^{r-1}\log(a+\mathbf{m}\cdot\mathbf{w}+2)}
 {(a+\mathbf{m}\cdot\mathbf{w})^{2\sigma-1}} \\
 &\ll
 \sum_{0\le m_1,\dots,m_r\le T}
 \frac{\log(2+m_1+\cdots+m_r)}
 {(1+m_1+\cdots+m_r)^{2\sigma-r}}.
\end{align*}
Let $N=m_1+\cdots+m_r$. Since
\[
\#\{(m_1,\dots,m_r)\in\mathbb Z_{\ge0}^r \mid \ m_1+\cdots+m_r=N\}
=
\binom{N+r-1}{r-1}
\ll N^{r-1},
\]
we find
\begin{align*}
 W_1
 &
 \ll
 \sum_{0 \le N \ll T} 
 \frac{\log(1+N)}{(1+N)^{2\sigma-r}} \binom{N+r-1}{r-1}
 %%%
 \ll
 \sum_{0 \le N \ll T}
 (1+N)^{2r-2\sigma-1} \log(1+N) \\
%
%The sum evaluates as
 &
 \ll
 \begin{cases}
 T^{2r-2\sigma} \log T
 &\text{if } r-1<\sigma< r, \\
 (\log T)^2 
 &\text{if } \sigma=r.
\end{cases}
\end{align*}
And then, by the evaluations of $W_1$ and $W_2$, $S_2$ evaluates as
\begin{align*}
 S_2 \ll W_1+W_2
 \ll
 \begin{cases}
  T^{2r-2\sigma} \log T 
 &\text{if } r-1<\sigma< r, \\
  (\log T)^2
 &\text{if } \sigma=r.
 \end{cases}
\end{align*}
Thus, using the evaluations of $S_1$ and $S_2$, we obtain
\begin{align}
 \int_1^T|\Sigma(s)|^2 dt
=\tilde{\zeta_{r}}(\sigma,a,\mathbf{w})T
+ \begin{cases}
 O(T^{2r-2\sigma} \log T) 
 &\text{if } r-1<\sigma< r, \\
  O((\log T)^2) 
 &\text{if } \sigma=r.
 \end{cases}  \label{int_Sigma^2}
\end{align}

Furthermore, we obtain from \eqref{sigma+O} that
\begin{align*}
 &\int_1^T |\zeta_r(\sigma+it,a,\mathbf{w})|^2 \,dt\\
 &=\int_1^T |\Sigma(s)+O(t^{r-1-\sigma})|^2 \,dt\\
 &=\int_1^T |\Sigma(s)|^2 \,dt 
  +O\left(\int_1^T |\Sigma(s)| t^{r-1-\sigma} dt\right)+O \left(\int_1^T t^{2r-2-2\sigma} dt \right).
\end{align*}
We see that the third term on the right-hand side is estimated as
\begin{equation}
 \int_1^T t^{2r-2-2\sigma} dt
 \ll
 \begin{cases}
 T^{2r-2\sigma-1} 
 &\text{if } r-1<\sigma< r-1/2, \\
 \log T 
 &\text{if } \sigma=r-1/2, \\
 1
 &\text{if } r-1/2<\sigma\le r.
\end{cases} \label{int_t^2sigma-r-2}
\end{equation}
Also, using the Cauchy-Schwarz inequality for the second terms on the right-hand side, we see that
\begin{align}
 \int_1^T |\Sigma(s)|t^{r-1-\sigma} dt 
 &\le \left(\int_1^T |\Sigma(s)|^2 dt\right)^{1/2} \left(\int_1^T t^{2r-2-2\sigma} dt\right)^{1/2} \nonumber \\
 & \ll
 \begin{cases}
  T^{2r-2\sigma-1/2}(\log{T})^{1/2} 
 &\text{if } r-1<\sigma< r-1/2, \\
  T^{1/2}\log{T} 
 &\text{if } \sigma=r-1/2, \\
  T^{1/2}
 &\text{if } r-1/2<\sigma\le r.
\end{cases}\label{C-S_ineq}
\end{align}
Therefore, since \eqref{int_Sigma^2},\eqref{int_t^2sigma-r-2} and \eqref{C-S_ineq}, we have
\begin{align*}
 \int_1^T |\zeta_r(\sigma+it,a,\mathbf{w})|^2 \,dt 
 =\tilde{\zeta
 _{r}}(\sigma,a,\mathbf{w})T  
 +
 \begin{cases}
 O(T^{2r-2\sigma} \log T) 
 &\text{if } r-1/2<\sigma \le r-1/4, \\
  O(T^{1/2}) 
 &\text{if } r-1/4<\sigma \le r,
 \end{cases}
\end{align*}
and
%, since \eqref{int_Sigma^2},\eqref{int_t^2sigma-r-2} and \eqref{C-S_ineq}, we have
\begin{align*}
 \int_1^T |\zeta_r(\sigma+it,a,\mathbf{w})|^2 \,dt 
 \ll T^{2r-2\sigma} \log T \quad \text{if } r-1<\sigma\le r-1/2.
\end{align*}
This completes the proof of Theorem \ref{main3}.
\end{proof}

%%%%%%%%%%%%%%%%%%%%%%%%%%%%%%%%%%%%%%%%%%%%%%%%%%%%%%%%%%%%
\section*{Acknowledgments}
The authors thank the anonymous referee for helpful comments, which helped improve the accuracy of our proofs.
This work was supported by JSPS KAKENHI Grant Number JP22K13897.

%%%%%%%%%%%%%%%%%%%%%%%%%%%%%%%%%%%%%%%%%%%%%%%%%%%%%%%%%%%%
\bibliographystyle{amsalpha}
\bibliography{References}
\end{document}